\documentclass{article}

\usepackage{amsmath,amssymb,amsthm,mathrsfs}
\usepackage{hyperref}

\newtheorem{theorem}{Theorem}[section]
\newtheorem{proposition}[theorem]{Proposition}
\newtheorem{lemma}[theorem]{Lemma}

\theoremstyle{definition}\newtheorem{example}[theorem]{Example}
\theoremstyle{definition}\newtheorem{definition}[theorem]{Definition}
\theoremstyle{definition}\newtheorem{remark}[theorem]{Remark}

\DeclareMathOperator*{\argmin}{argmin}

\def\di{\displaystyle}

\def\R{\mathbb{R}}
\def\N{\mathbb{N}}
\def\H{\mathrm{H}}
\def\I{\mathrm{I}}
\def\M{\mathrm{M}}
\def\D{\mathrm{D}}
\def\RR{\mathrm{R}}
\def\K{\mathrm{K}}
\def\VI{\mathrm{VI}}
\def\P{\mathrm{P}}
\def\T{\mathcal{DR}}
\def\TT{\overline{\T}}
\def\BF{\mathcal{FB}}
\def\BFF{\overline{\BF}}

\def\dom{\mathrm{dom}}
\def\prox{\mathrm{prox}}
\def\proj{\mathrm{proj}}
\def\Fix{\mathrm{Fix}}
\def\Sol{\mathrm{Sol}}
\def\int{\mathrm{int}}
\def\cl{\mathrm{cl}}

\usepackage{tikz}
\usetikzlibrary{decorations.pathreplacing}
\usetikzlibrary{decorations.pathmorphing}
\usetikzlibrary{arrows,shapes}
\usepackage{pstricks,pst-grad}

\newrgbcolor{white}{1. 1. 1.}
\newrgbcolor{lightblue}{0. 0. 0.80}
\newrgbcolor{whiteblue}{.80 .80 1.}
\newrgbcolor{lightred}{.85 0. 0.}
\newrgbcolor{lightred2}{.85 0.40 0.40}
\newrgbcolor{whitered}{0.85 0.45 0.45}
\newrgbcolor{lightgreen}{0. 0.50 0.}
\newrgbcolor{whitegreen}{0.60 0.95 0.40}
\newrgbcolor{lightbrown}{0.90 0.35 0.05}
\newrgbcolor{whitebrown}{0.95 0.55 0.20}
\definecolor{bovert}{RGB}{0,154,0}

\title{On a decomposition formula for the proximal operator of the sum of two convex functions}

\author{Samir Adly\footnote{Institut de recherche XLIM. UMR CNRS 7252. Universit\'e de Limoges, France. \texttt{samir.adly@unilim.fr}},
Lo\"ic Bourdin\footnote{Institut de recherche XLIM. UMR CNRS 7252. Universit\'e de Limoges, France. \texttt{loic.bourdin@unilim.fr}}, 
Fabien Caubet\footnote{Institut de Math\'ematiques de Toulouse. UMR CNRS 5219. Universit\'e de Toulouse, France. \texttt{fabien.caubet@math.univ-toulouse.fr}}
}

\begin{document}

\maketitle

\begin{abstract}
The main result of the present theoretical paper is an original decomposition formula for the proximal operator of the sum of two proper, lower semicontinuous and convex functions~$f$ and~$g$. For this purpose, we introduce a new operator, called {\it $f$-proximal operator of~$g$} and denoted by $\prox^f_g$, that generalizes the classical notion. Then we prove the decomposition formula~$\prox_{f+g} = \prox_f \circ \prox^f_g$. After collecting several properties and characterizations of~$\prox^f_g$, we prove that it coincides with the fixed points of a generalized version of the classical Douglas-Rachford operator. This relationship is used for the construction of a weakly convergent algorithm that computes numerically this new operator~$\prox^f_g$, and thus, from the decomposition formula, allows to compute numerically~$\prox_{f+g}$. It turns out that this algorithm was already considered and implemented in previous works, showing that~$\prox^f_g$ is already present (in a hidden form) and useful for numerical purposes in the existing literature. However, to the best of our knowledge, it has never been explicitly expressed in a closed formula and neither been deeply studied from a theoretical point of view. The present paper contributes to fill this gap in the literature. Finally we give an illustration of the usefulness of the decomposition formula in the context of sensitivity analysis of linear variational inequalities of second kind in a Hilbert space.
\end{abstract}


\medskip


\tableofcontents

\section{Introduction, notations and basics}

\subsection{Introduction}

The \textit{proximal operator} (also known as \textit{proximity operator}) of a proper, lower semicontinuous, convex and extended-real-valued function was first introduced by~J.-J.~Moreau in~1962 in \cite{moreau62bis,moreau65} and can be viewed as an extension of the projection operator on a nonempty closed and convex subset of a Hilbert space. This wonderful tool plays an important role, from both theoretical and numerical points of view, in applied mathematics and engineering sciences. This paper fits within the wide theoretical literature dealing with the proximal operator. For the rest of this introduction, we use standard notations of convex analysis. For the reader who is not acquainted with convex analysis, we refer to Section~\ref{secnot} for notations and basics.


\paragraph{\textbf{Motivations from a sensitivity analysis.}} 

The present work was initially motivated by the sensitivity analysis, with respect to a nonnegative parameter $t \geq 0$, of a parameterized linear variational inequality of second kind in a Hilbert space~$\H$, with a corresponding function~$h \in \Gamma_{0}(\H)$, where $\Gamma_0(\H)$ is the set of proper, lower semicontinuous and convex functions from $\mathrm{H}$ into $\mathbb{R} \cup \left\lbrace + \infty \right\rbrace$. 
More precisely, for all $t \geq 0$, we consider the problem of finding $u(t) \in \H$ such that
\begin{equation} \label{VIintro}
\langle u(t) , z - u(t) \rangle + h(z) - h(u(t)) \geq \langle r(t) , z - u(t) \rangle ,
\end{equation}
for all $z \in \H$, where $r : \R^+ \to \H$ is assumed to be given and smooth enough. 
In that framework, the solution~$u(t) \in \H$ (which depends on the parameter $t$) can be expressed in terms of the proximal operator of $h$ denoted by $\prox_h$. Precisely it holds that $u(t)=\prox_{h} (r(t))$ for all $t \geq 0$. As a consequence, the differentiability of~$u(\cdot)$ at $t=0$ is strongly related to the regularity of $\prox_h$. If $h$ is a smooth function, one can easily compute (from the classical inverse mapping theorem for instance) the differential of $\prox_{h}$, and then the sensitivity analysis can be achieved. In that smooth case, note that the variational inequality~\eqref{VIintro} can actually be reduced to an equality. On the other hand, if $h=\iota_{\K}$ is the indicator function of a nonempty closed and convex subset $\K \subset \H$, then $\prox_{h} = \proj_{\K}$ is the classical projection operator on $\K$. In that case, the work of F.~Mignot in~\cite[Theorem~2.1 p.145]{Mig76} (see also the work of A.~Haraux in~\cite[Theorem~2 p.620]{haraux}) provides an asymptotic expansion of $\prox_h = \proj_{\K}$ and permits to obtain a differentiability result on $u(\cdot)$ at~$t=0$. 

\medskip

In a parallel work (in progress) of the authors on some shape optimization problems with unilateral contact and friction, the considered variational inequality~\eqref{VIintro}
involves the sum of two functions. Precisely, $h=f+g$ where $f = \iota_\K$ ($\K$ being a nonempty closed and convex set of constraints), and where~$g \in \Gamma_{0}(\H)$ is a smooth function (derived from the regularization of the friction functional in view of a numerical treatment). Despite the regularity of~$g$, note that the variational inequality~\eqref{VIintro} cannot be reduced to an equality due to the presence of the constraint set~$\K$. In that framework, in order to get an asymptotic expansion of~$\prox_{h} = \prox_{f+g}$, a first and natural strategy would be to look for a convenient explicit expression of $\prox_{f+g}$ in terms of $\prox_f$ and~$\prox_g$. Unfortunately, this theoretical question still remains an open challenge in the literature. Let us mention that Y.-L.~Yu provides in \cite{Yu13} some necessary and/or sufficient conditions on general functions $f$, $g \in \Gamma_0 (\H)$ under which~$\prox_{f+g} = \prox_f \circ \prox_g$. Unfortunately, as underlined by the author himself, these conditions are very restrictive and are not satisfied in most of cases (see, e.g., \cite[Example~2]{Yu13} for a counterexample). 

\medskip

Before coming to the main topic of this paper, we recall that a wide literature is already concerned with the sensitivity analysis of parameterized (linear and nonlinear) variational inequalities. We refer for instance to \cite{BenLions,haraux,QiuMag,Shapiro} and references therein. The results in are considered in very general frameworks. We precise that our original objective was to look for a simple and compact formula for the derivative~$u'(0)$ in the very particular case described above, that is, in the context of a linear variational inequality and with $h = f+g$ where~$f$ is an indicator function and $g$ is a smooth function. For this purpose, we were led to consider the proximal operator of the sum of two functions in $\Gamma_{0}(\H)$, 
to introduce a new operator and finally to prove the results presented in this paper.


\paragraph{\textbf{Introduction of the $f$-proximal operator and main result.}} 
Let us consider general functions $f$, $g \in \Gamma_0 (\H) $. In order to avoid trivialities, we will assume in the whole paper that $\dom (f) \cap \dom (g) \neq \emptyset$ when dealing with the sum $f+g$. 

\medskip

Section~\ref{sec2} is devoted to the introduction (see Definition~\ref{def1}) of a new operator $\prox^f_g : \H \rightrightarrows \H$ called \textit{$f$-proximal operator of $g$} and defined by
\begin{equation}\label{eqdefintro}
\prox^f_{g} := \left( \I + \partial g \circ \prox_{f} \right)^{-1}.
\end{equation}
This new operator can be seen as a generalization of $\prox_g$ in the sense that, if $f$ is constant for instance, then $\prox^f_g = \prox_g$. More general sufficient (and necessary) conditions under which~$\prox^f_g = \prox_g$ are provided in Propositions~\ref{prop3} and \ref{prop4}. We prove in Proposition~\ref{prop1} that the domain of $\prox^f_{g}$ satisfies $\D ( \prox^f_g) = \H $ if and only if $\partial(f+g)=\partial f + \partial g$. Note that $\prox^f_g$ is a set-valued operator \textit{a priori}. We provide in Proposition~\ref{prop2} some sufficient conditions under which $\prox^f_g$ is single-valued. Some examples illustrate all the previous results throughout the section (see Examples~\ref{ex1},~\ref{ex2},~\ref{ex2bis}, \ref{ex3} and \ref{ex4}). 

\medskip

Finally, if the additivity condition $\partial(f+g)=\partial f + \partial g$ is satisfied, the main result of the present paper (see Theorem~\ref{thm1}) is the original decomposition formula
\begin{equation}\label{proxfg}
\mathrm{prox}_{f+g} = \mathrm{prox}_{f} \circ \mathrm{prox}^f_{g}.
\end{equation}
It is well-known in the literature that obtaining a theoretical formula for~$\prox_{f+g}$ is not an easy task in general, even if $\prox_f$ and $\prox_g$ are known. We give a more precise description of the difficulty to obtain an {easy computable} formula of~$\prox_{f+g}$ in Appendix~\ref{app}, which claims that there is no closed formula, independent of $f$ and $g$, allowing to write $\prox_{f+g}$ as a linear combination of compositions of linear combinations of~$\I$, $\prox_f$, $\prox_g$, $\prox_f^{-1}$ and $\prox_g^{-1}$. In the decomposition formula~\eqref{proxfg}, it should be noted that the difficulty of computing~$\prox_{f+g}$ is only transferred to the computation of~$\prox^f_g$ which is not an easier task. Note that other rewritings, which are not suitable for an {easy computation} of~$\prox_{f+g}$ neither, can be considered such as
$$ \prox_{f+g} = ( \prox_f^{-1}+\prox_g^{-1} - \I )^{-1} = ( \prox_{2f}^{-1} + \prox_{2g}^{-1} )^{-1} \circ 2\I, $$
the second equality being provided in~\cite[Corollary 25.35 p.458]{BauCom17}. However we show in this paper that our decomposition formula~\eqref{proxfg} is of theoretical interest in order to prove in a concice and elegant way almost all other new statements of this paper, and also to recover in a simple way some well-known results (see Sections~\ref{sec33} and~\ref{sec41} for instance), making it central in our work. We provide an illustration of this feature in the next paragraph about the classical Douglas-Rachford algorithm. Moreover, as explained in the last paragraph of this introduction, we also prove in this paper the usefulness of the decomposition formula~\eqref{proxfg} in the context of sensitivity analysis of the variational inequality~\eqref{VIintro} (see Section~\ref{sec42}).


\paragraph{\textbf{Relationship with the classical Douglas-Rachford operator.}}

Recall that the proximal operator $\prox_{f+g}$ is strongly related to the minimization problem
$$ \argmin \, f+g , $$
since the set of solutions is exactly the set of fixed points of $\prox_{f+g}$ denoted by $\Fix (\prox_{f+g})$. In the sequel, we will assume that the above minimization problem admits at least one solution. The classical \textit{Douglas-Rachford operator}, introduced in \cite{DouglasRachford} and denoted here by $\T_{f,g}$ (see Section~\ref{sec3} for details), provides an algorithm $x_{n+1} = \T_{f,g} (x_n)$ that is weakly convergent to some $x^* \in \H$ satisfying 
$$ \prox_f (x^*) \in \argmin \, f+g . $$
Even if the \textit{Douglas-Rachford algorithm} is not a \textit{proximal point algorithm} in general, in the sense that $\T_{f,g}$ is not equal to $\prox_\varphi$ for some $\varphi \in \Gamma_0(\H)$ in general, it is a very powerful tool since it allows to solve the above minimization problem, requiring only the knowledge of $\prox_f$ and $\prox_g$. {We refer to~\cite[Section 28.3 p.517]{BauCom17} for more details}.

\medskip

Section~\ref{sec3} deals with the relations between the Douglas-Rachford operator $\T_{f,g}$ and the $f$-proximal operator $\prox^f_g$ introduced in this paper. Precisely, we prove in Proposition~\ref{prop5} that
$$ \prox^f_g (x) = \Fix \left( \TT_{f,g}(x,\cdot) \right) ,$$
for all~$x \in \H$, where $\TT_{f,g}(x,\cdot)$ denotes a $x$-dependent generalization of the classical Douglas-Rachford operator $\T_{f,g}$, in the sense that $\T_{f,g} (y) = \TT_{f,g} ( \prox_f(y),y)$ for all $y \in \H$. We refer to Section~\ref{sec3} for the precise definition of~$\TT_{f,g} (x,\cdot)$ that only depends on the knowledge of $\prox_f$ and $\prox_g$.

\medskip

Let us show that the above statements, in particular the decomposition formula~\eqref{proxfg}, allow to recover in a concise way the well-known inclusion
\begin{equation}\label{eq78654}
\prox_f \left( \Fix \left( \T_{f,g} \right) \right) \subset  \argmin \, f+g = \Fix ( \prox_{f+g} ).
\end{equation}
Indeed, if $x^* \in \Fix (\T_{f,g})$, then $x^* \in \Fix ( \TT_{f,g}(\prox_f(x^*),\cdot) ) = \prox^f_g ( \prox_f(x^*))$. From the decomposition formula~\eqref{proxfg}, we conclude that
$$ \prox_f (x^*) = \prox_f \circ \prox^f_g ( \prox_f (x^*) ) = \prox_{f+g} ( \prox_f(x^*)). $$
This proof of only few lines is an illustration of the theoretical interest of the decomposition formula~\eqref{proxfg}. Note that the above inclusion~\eqref{eq78654} is, as well-known, an equality (see Section~\ref{sec33} and Proposition~\ref{proppourintro} for details).

\medskip

The $f$-proximal operator $\prox^f_g$ introduced in this paper is also of interest from a numerical point of view. Indeed, if~$x \in \D (\prox^f_g)$, we prove in Theorem~\ref{thm2} that the fixed-point algorithm~$y_{k+1} = \TT_{f,g}(x,y_k)$, denoted by~\eqref{eqalgo}, weakly converges to some~$y^* \in \prox^f_g(x)$. Moreover, if the additivity condition~$\partial (f+g) = \partial f + \partial g$ is satisfied, we get from the decomposition formula~\eqref{proxfg} that~$\prox_f (y^*) = \prox_{f+g} (x)$. In that situation, we conclude that Algorithm~\eqref{eqalgo} allows to compute numerically~$\prox_{f+g}(x)$ with the only knowledge of~$\prox_f$ and~$\prox_g$. It turns out that Algorithm~\eqref{eqalgo} was already considered, up to some translations, and implemented in previous works (see, e.g., \cite[Algorithm~3.5]{combdungvu}), showing that the $f$-proximal operator $\prox^f_g$ is already present (in a hidden form) and useful for numerical purposes in the existing literature. However, to the best of our knowledge, it has never been explicitly expressed in a closed formula such as~\eqref{eqdefintro} and neither been deeply studied from a theoretical point of view. The present paper contributes to fill this gap in the literature.

\paragraph{\textbf{Some other applications and forthcoming works.}}

Section~\ref{sec4} can be seen as a conclusion of the paper. Its aim is to provide a glimpse of some other applications of our main result (Theorem~\ref{thm1}) and to raise open questions for forthcoming works. This section is splitted into two parts.

\medskip

In Section~\ref{sec41} we consider the framework where $f$, $g \in \Gamma_0(\H)$ with $g$ differentiable on~$\H$. In that framework, we prove from the decomposition formula~\eqref{proxfg} that $\prox_{f+g}$ is related to the classical \textit{Forward-Backward operator} {(see \cite[Section 10.3 p.191]{ComPes11} for details)} denoted by $\BF_{f,g}$. Precisely, we prove in Proposition~\ref{prop6} that
$$ \prox_{f+g} (x) = \Fix \left( \BFF_{f,g}(x,\cdot) \right) ,$$
for all~$x \in \H$, where~$\BFF_{f,g}(x,\cdot)$ denotes a $x$-dependent generalization of the classical Forward-Backward operator $\BF_{f,g}$. We refer to Section~\ref{sec41} for the precise definition of~$\BFF_{f,g} (x,\cdot)$ that only depends on the knowledge of $\prox_f$ and~$\nabla g$. From this point, one can develop a similar strategy as in Section~\ref{sec3}. Precisely, for all $x \in \H$, one can consider the algorithm $y_{k+1} = \BFF_{f,g} ( x , y_k )$, denoted by~\eqref{eqalgo3}, in order to compute numerically $\prox_{f+g}(x)$, with the only knowledge of $\prox_f$ and $\nabla g$. Convergence proof (under some assumptions on $f$ and~$g$) of~\eqref{eqalgo3} should be the topic of a future work.

\medskip

In Section~\ref{sec42} we turn back to our initial motivation, namely the sensitivity analysis of the parameterized variational inequality~\eqref{VIintro}. Precisely, under some assumptions (see Proposition~\ref{PropSensitivity} for details), we derive from the decomposition formula~\eqref{proxfg} that if
$$ u(t) := \prox_{f+g} (r(t)), $$
for all $t \geq 0$, where $f := \iota_\K$ (where $\K \subset \H$ is a nonempty closed convex subset) and where $g \in \Gamma_0 (\H)$ and $r : \R^+ \to \H$ are smooth enough, then
$$ u'(0) = \prox_{\varphi_f + \psi_g} (r'(0)), $$
where $\varphi_f := \iota_C$ (where $C$ is a nonempty closed convex subset of $\H$ related to $\K$) and where $\psi_g (x) := \frac{1}{2} \langle \D^2 g(u(0))(x),x \rangle $ for all $x \in \H$. It should be noted that the assumptions of Proposition~\ref{PropSensitivity} are quite restrictive, raising open questions about their relaxations {(see Remark~\ref{remfinal})}. This also should be the subject of a forthcoming work.

\subsection{Notations and basics}\label{secnot}

In this section we introduce some notations available throughout the paper and we recall some basics of convex analysis. We refer to standard books like \cite{BauCom17,JBHU,Rock3} and references therein.

\medskip

Let $\H$ be a real Hilbert space and let $\langle \cdot , \cdot \rangle$ (resp. $\Vert \cdot \Vert$) be the corresponding scalar product (resp. norm). For every subset $S$ of $\H$, we denote respectively by $\int (S)$ and~$\cl(S)$ its interior and its closure. In the sequel we denote by $\I : \H \to \H$ the identity operator and by $L_x : \H \to \H$ the affine operator defined by 
$$
L_x (y) := x-y ,
$$
for all $x$, $y \in \H$.

\medskip

For a set-valued map $A : \H \rightrightarrows \H$, the {\it domain} of $A$ is given by 
$$
\D (A) := \{ x \in \H \mid A(x) \neq \emptyset \}.
$$
We denote by $A^{-1} : \H \rightrightarrows \H$ the set-valued map defined by
$$ A^{-1} (y) := \{ x \in \H \mid y \in A(x) \}, $$
for all $y \in \H$. Note that $y \in A(x)$ if and only if $x \in A^{-1} (y)$, for all $x$, $y \in \H$. The {\it range} of $A$ is given by 
$$
\RR (A) := \{ y \in \H \mid A^{-1} (y) \neq \emptyset \} = \D (A^{-1}).
$$
We denote by $\Fix (A)$ the set of all fixed points of $A$, that is, the set given by
$$
\Fix (A) := \{ x \in \H \mid x \in A(x) \}.
$$
Finally, if $A(x)$ is a singleton for all $x \in \D (A)$, we say that $A$ is {\it single-valued}. 


\medskip

For all extended-real-valued functions $g : \H \to \R \cup \{ +\infty \}$, the {\it domain} of $g$ is given by 
$$
\dom (g) := \{ x \in \H \mid g(x) < +\infty \}.
$$
Recall that $g$ is said to be \textit{proper} if $\dom (g) \neq \emptyset$. 

\medskip

Let $g : \H \to \R \cup \{ +\infty \}$ be a proper extended-real-valued function. We denote by~$g^\ast : \H \to \R \cup \{ +\infty \}$ the {\it conjugate} of $g$ defined by
$$ 
g^\ast (y) := \sup_{z \in \H} \; \{ \langle y , z \rangle - g(z) \}, 
$$
for all $y \in \H$. Clearly $g^\ast$ is lower semicontinuous and convex.

\medskip

We denote by $\Gamma_0 (\H)$ the set of all extended-real-valued functions $g : \H \to \R \cup \{ +\infty \}$ that are proper, lower semicontinuous and convex. If $g \in \Gamma_0 (\H)$, we recall that $g^\ast \in \Gamma_0 (\H)$ and that the \textit{Fenchel-Moreau equality}~$g^{\ast \ast} = g$ holds. For all~$g \in \Gamma_0 (\H)$, we denote by $\partial g : \H \rightrightarrows \H$ the {\it Fenchel-Moreau subdifferential} of~$g$ defined by
$$ \partial g (x) := \{ y \in \H \mid \langle y , z - x \rangle \leq g(z) - g(x), \; \forall z \in \H  \}, $$
for all $x \in \H$. It is easy to check that $\partial g$ is a monotone operator and that, for all~$x \in \H$, $0 \in \partial g(x)$ if and only if $x \in \argmin g$. Moreover, for all $x$, $y \in \H$, it holds that $y \in \partial g (x) $ if and only if~$x \in \partial g^\ast (y)$. Recall that, if $g$ is differentiable on $\H$, then~$\partial g (x) = \{ \nabla g (x) \}$ for all $x \in \H$. 

\medskip

Let $A : \H \to \H$ be a single-valued operator defined everywhere on $\H$, and let $g \in \Gamma_0 (\H)$. We denote by~$\VI (A,g)$ the variational inequality which consists of finding $y \in \H$ such that
\begin{equation*}
-A(y) \in \partial g (y),
\end{equation*}
or equivalently,
\begin{equation*}
\langle A(y) , z - y \rangle + g(z) - g(y) \geq 0,
\end{equation*}
for all $z \in \H$. Then we denote by $\Sol_\VI (A,g)$ the set of solutions of~$\VI (A,g)$. Recall that if $A$ is Lipschitzian and strongly monotone, then~$\VI (A,g)$ admits a unique solution, \textit{i.e.} $\Sol_\VI (A,g)$ is a singleton.

\medskip

Let $g \in \Gamma_0 (\H)$. The classical {\it proximal operator} of $g$ is defined by
$$ 
\prox_g := (\I +\partial g )^{-1}. 
$$
Recall that $\prox_g$ is a single-valued operator defined everywhere on $\H$. Moreover, it can be characterized as follows:
$$ 
\prox_g(x) = \argmin \Big(g+ \dfrac{1}{2} \Vert \cdot - x \Vert^2  \Big) = \Sol_\VI (-L_x , g )  ,
$$
for all $x \in \H$. It is also well-known that
$$ 
\Fix ( \prox_g ) = \argmin g. 
$$
The classical {\it Moreau's envelope} $\M_g : \H \to \R$ of $g$ is defined by
$$ \M_g (x) := \min \Big( g+\dfrac{1}{2} \Vert \cdot - x \Vert^2 \Big), $$
for all $x \in \H$. Recall that $\M_g$ is convex and differentiable on $\H$ with $\nabla \M_g = \prox_{g^\ast}$. Let us also recall the classical Moreau's decompositions
$$ \prox_g +\prox_{g^\ast} = \I \qquad \text{and} \qquad \M_g + \M_{g^\ast} = \frac{1}{2} \Vert \cdot \Vert^2. $$

Finally, it is well-known that if $g = \iota_\K$ is the {\it indicator function} of a nonempty closed and convex subset $\K$ of $\H$, that is, $\iota_\K (x) = 0$ if $x \in \K$ and $\iota_\K (x) = +\infty$ if not, then~$\prox_g = \proj_\K$, where~$\proj_\K$ denotes the classical projection operator on $\K$.

\section{The $f$-proximal operator}\label{sec2}

\subsection{Definition and main result}

Let $f$, $g \in \Gamma_0 (\H)$. In this section we introduce (see Definition~\ref{def1}) a new operator denoted by $\prox^f_g$, generalizing the classical proximal operator $\prox_g$. Assuming that $\dom (f) \cap \dom (g) \neq \emptyset$, and under the additivity condition $\partial (f+g) = \partial f + \partial g $, we prove in Theorem~\ref{thm1} that $\prox_{f+g}$ can be written as the composition of $\prox_f$ with $\prox^f_g$. 

\begin{definition}[$f$-proximal operator]\label{def1}
Let $f$, $g \in \Gamma_0 (\H)$. The \textit{$f$-proximal operator} of $g$ is the set-valued map $\prox^f_g : \H \rightrightarrows \H$ defined by
\begin{equation}\label{eqdefdef}
\prox^f_g := (\I +\partial g \circ \prox_f)^{-1}. 
\end{equation}
\end{definition}

Note that $\prox^f_g$ can be seen as a generalization of $\prox_g$ since~$\prox^c_g = \prox_g$ for all constant $c \in \R$. 

\begin{example}\label{ex1}
Let us assume that $\H = \R$. We consider $f = \iota_{ [-1,1] }$ and $g(x) = \vert x \vert$ for all $x \in \R$. In that case we obtain that $\partial g \circ \prox_f = \partial g $ and thus $\prox^f_g = \prox_g$.
\end{example}

Example~\ref{ex1} provides a simple situation where $\prox^f_g = \prox_g$ while $f$ is not constant. We provide in Propositions~\ref{prop3} and \ref{prop4} some general sufficient (and necessary) conditions under which $\prox^f_g = \prox_g$. 

\begin{example}\label{ex2}
Let us assume that $\H = \R$. We consider $f = \iota_{\{ 0 \}}$ and $g(x) = \vert x \vert$ for all $x \in \R$. In that case we obtain that $\partial g \circ \prox_f (x) = [-1,1]$ for all $x \in \R$. As a consequence $\prox^f_g (x) = [x-1,x+1]$ for all $x \in \R$. See Figure~\ref{fig1} for graphical representations of $\prox_g$ and $\prox^f_g$ in that case.
\begin{figure}[h]
\begin{center}
\begin{tikzpicture}
\draw[fill=gray, fill opacity=0.2] (-3,-2) -- (2,3) -- (3,3) -- (3,2) -- (-2,-3) -- (-3,-3) -- cycle;
\draw[dashed] (-3,-3) grid (3,3);
\draw[->] (-3,0) -- (3,0);
\draw[->] (0,-3) -- (0,3);
\node at (-0.2,-0.3) {$0$};
\draw[color=black] (-3,-2) -- (2,3); 
\draw[color=black] (-2,-3) -- (3,2); 
\draw[line width=0.1cm, color=black] (-3,-2) -- (-1,0) -- (1,0)  -- (3,2);

\node at (2.5,0.75) {$\prox_g$};
\node at (1.5,1.5) {$\prox^f_g$};

\end{tikzpicture}
\caption{Example~\ref{ex2}, graph of $\prox_g$ in bold line, and graph of $\prox^f_g$ in gray.}\label{fig1}
\end{center}
\end{figure}
\end{example}

\begin{example}\label{ex2bis}
Let us assume that $\H = \R$. We consider $f(x) = g(x) = \vert x \vert$ for all $x \in \R$. In that case we obtain that $\partial g \circ \prox_f (x) = -1$ for all $x < -1$, $\partial g \circ \prox_f (x) = [-1,1]$ for all $x \in [-1,1]$ and $\partial g \circ \prox_f (x) = 1$ for all $x > 1$. As a consequence $\prox^f_g (x) = x+1$ for all $x \leq -2$, $\prox^f_g (x) = [-1,x+1]$ for all $x \in [-2,0]$, $\prox^f_g (x) = [x-1,1]$ for all $x \in [0,2]$ and $\prox^f_g (x) = x-1$ for all $x \geq 2$. See Figure~\ref{fig2} for graphical representations of $\prox_g$ and $\prox^f_g$ in that case.
\begin{figure}[h]
\begin{center}
\begin{tikzpicture}
\draw[fill=gray, fill opacity=0.2] (-2,-1) -- (0,1) -- (2,1) -- (0,-1) -- cycle;
\draw[dashed] (-3,-3) grid (3,3);
\draw[->] (-3,0) -- (3,0);
\draw[->] (0,-3) -- (0,3);
\node at (-0.2,-0.3) {$0$};
\draw[line width=0.1cm, color=black] (-3,-2) -- (-1,0) -- (1,0)  -- (3,2); 

\node at (2.5,0.75) {$\prox_g$};
\node at (0.5,0.5) {$\prox^f_g$};

\end{tikzpicture}
\caption{Example~\ref{ex2bis}, graph of $\prox_g$ in bold line, and graph of $\prox^f_g$ in gray.}\label{fig2}
\end{center}
\end{figure}
\end{example}

Examples~\ref{ex2} and \ref{ex2bis} provide simple illustrations where $\prox^f_g$ is not single-valued. In particular it follows that $\prox^f_g$ cannot be written as a proximal operator $\prox_\varphi$ for some~$\varphi \in \Gamma_0 (\H)$. We provide in Proposition~\ref{prop2} some sufficient conditions under which $\prox^f_g$ is single-valued. Moreover, Examples~\ref{ex2} and \ref{ex2bis} provide simple situations where $\partial g \circ \prox_f$ is not a monotone operator. As a consequence, it may be possible that $\D (\prox^f_g) \varsubsetneq \H$. In the next proposition, a necessary and sufficient condition under which $\D (\prox^f_g) = \H$ is derived.

\begin{proposition}\label{prop1}
Let $f$, $g \in \Gamma_0 (\H)$ such that $\dom (f) \cap \dom (g) \neq \emptyset$. It holds that $ \D ( \prox^f_g ) = \H $ if and only if the additivity condition
\begin{equation}\label{eqcondition1}\tag{$\mathrm{C}_1$}
\partial (f+g) = \partial f + \partial g ,
\end{equation}
is satisfied.
\end{proposition}

\begin{proof}
We first assume that $\partial (f+g) = \partial f + \partial g$. Let $x \in \H$. Defining $w = \prox_{f+g} (x) \in \H$, we obtain that $x \in w + \partial (f+g)(w) = w + \partial f (w) + \partial g (w)$. Thus, there exist $w_f \in \partial f (w)$ and~$w_g \in \partial g(w)$ such that $x = w +w_f +w_g$. We define~$y = w + w_f \in w + \partial f (w)$. In particular we have $w = \prox_f (y)$. Moreover we obtain $x=y+w_g \in y + \partial g (w) = y + \partial g ( \prox_f (y) )$. We conclude that $y \in \prox^f_g (x)$. 

\medskip

Without any additional assumption and directly from the definition of the subdifferential, one can easily see that the inclusion $\partial f (w) + \partial g (w) \subset \partial (f+g) (w)$ is always satisfied for every $w \in \H$. Now~let us assume that $ \D ( \prox^f_g ) = \H $. Let $w \in \H$ and let~$z \in \partial (f+g) (w)$. We consider $x = w + z \in w + \partial (f+g) (w)$. In particular it holds that $w = \prox_{f+g} (x)$. Since $\D (\prox^f_g) = \H$, there exists $y \in \prox^f_g(x)$ and thus it holds that $x \in y + \partial g ( \prox_f (y))$. Moreover, since $y \in \prox_f (y) + \partial f ( \prox_f (y))$, we get that $x \in \prox_f (y) + \partial f ( \prox_f (y)) + \partial g ( \prox_f (y)) \subset \prox_f (y) + \partial (f+g) ( \prox_f (y))$. Thus it holds that $\prox_f (y) = \prox_{f+g} (x) = w$. Moreover, since $x \in \prox_f (y) + \partial f ( \prox_f (y)) + \partial g ( \prox_f (y))$, we obtain that $x \in w + \partial f ( w ) + \partial g ( w )$. We have proved that~$z = x- w \in \partial f ( w ) + \partial g ( w )$. This concludes the proof.
\end{proof}

In most of the present paper, we will assume that Condition~\eqref{eqcondition1} is satisfied. It is not our aim here to discuss the weakest qualification condition ensuring that condition. {A wide literature already deals with this topic (see, e.g., \cite{AttouchBrezis,EkelandTemam,Rock3}).} However, we recall in the following remark the classical sufficient condition of Moreau-Rockafellar under which Condition~\eqref{eqcondition1} holds true {(see, e.g., \cite[Corollary 16.48 p.277]{BauCom17})}, and we provide a simple example where Condition~\eqref{eqcondition1} does not hold and $\D(\prox^f_g) \varsubsetneq \H$.

\begin{remark}[Moreau-Rockafellar theorem]\label{remdsum}
Let $f$, $g \in \Gamma_0 (\H)$ such that $\dom (f) \cap \int (\dom (g)) \neq \emptyset$. Then~$\partial (f+g) = \partial f + \partial g$.
\end{remark}


\begin{example}\label{ex3}
Let us assume that $\H = \R$. We consider $f = \iota_{ \R^- }$ and $g(x) =\iota_{\R^+} (x) -\sqrt{x} $ for all $x \in \R$. In that case, one can easily check that $\dom (f) \cap \dom (g) = \{ 0 \} \neq \emptyset$, $\partial f (0) + \partial g (0) = \emptyset \varsubsetneq \R = \partial (f+g) (0)$ and~$\D(\prox^f_g) = \emptyset \varsubsetneq \H$.
\end{example}

We are now in position to state and prove the main result of the present paper.

\begin{theorem}\label{thm1}
Let $f$, $g \in \Gamma_0 (\H)$ such that $\dom (f) \cap \dom (g) \neq \emptyset$. If $\partial (f+g) = \partial f + \partial g$, then the decomposition formula
\begin{equation}\label{eqdecomp}
\prox_{f+g} = \prox_f \circ \prox^f_g
\end{equation}
holds true. In other words, for every $x \in \H$, we have $ \prox_{f+g} (x) =\prox_f(z) $ for all $z \in \prox^f_g (x)$.
\end{theorem}

\begin{proof}
Let $x \in \H$ and let $y \in \prox^f_g (x)$ constructed as in the first part of the proof of Proposition~\ref{prop1}. In particular it holds that $\prox_f (y) = \prox_{f+g}(x)$. Let $z \in \prox^f_g (x)$. We know that $x-y \in \partial g (\prox_f (y))$ and $x-z \in \partial g ( \prox_f (z))$. Since $\partial g$ is a monotone operator, we obtain that 
$$ \langle (x-y)-(x-z) , \prox_f (y) - \prox_f (z) \rangle \geq 0 .$$
From the cocoercivity (see for instance \cite[Definition 4.10 p.72]{BauCom17}) of the proximal operator, we obtain that 
$$ 0 \geq \langle y-z , \prox_f (y) - \prox_f (z) \rangle \geq \Vert \prox_f (y) - \prox_f (z) \Vert^2 \geq 0 .$$
We deduce that $\prox_f (z) = \prox_f(y) = \prox_{f+g}(x)$. The proof is complete.
\end{proof}

\begin{remark}
Let $f$, $g \in \Gamma_0 (\H)$ with $\dom (f) \cap \dom (g) \neq \emptyset$ and such that $\partial (f+g) = \partial f + \partial g$ and let $x \in \H$. Theorem~\ref{thm1} states that, even if $\prox^f_g(x)$ is not a singleton, all elements of $\prox^f_g (x)$ has the same value through the proximal operator $\prox_f$, and this value is equal to~$\prox_{f+g}(x)$. 
\end{remark}

\begin{remark}
Let $f$, $g \in \Gamma_0 (\H)$ with $\dom (f) \cap \dom (g) \neq \emptyset$. Note that the additivity condition $\partial (f+g) = \partial f + \partial g$ is not only sufficient, but also necessary for the validity of the equality $\prox_{f+g} = \prox_f \circ \prox^f_g$. Indeed, from Proposition~\ref{prop1}, if $\partial f + \partial g \varsubsetneq \partial (f+g)$, then there exists $x \in \H$ such that $\prox^f_g (x) = \emptyset$ and thus $\prox_{f+g} (x) \neq \prox_f \circ \prox^f_g (x)$.
\end{remark}

\begin{remark}\label{remRR}
Let $f$, $g \in \Gamma_0 (\H)$ with $\dom (f) \cap \dom (g) \neq \emptyset$ and such that $\partial (f+g) = \partial f + \partial g$. From Theorem~\ref{thm1}, we deduce that $\RR ( \prox_{f+g} ) \subset \RR (\prox_f) \cap \RR ( \prox_g )$. If the additivity condition $\partial (f+g) = \partial f + \partial g$ is not satisfied, this remark does not hold true anymore. Indeed, with the framework of Example~\ref{ex3}, we have $\RR (\prox_{f+g} ) = \{ 0 \}$ while $0 \notin \RR (\prox_g)$.
\end{remark}

\begin{example}\label{ex9809}
Following the idea of Y.-L.~Yu in \cite[Example~2]{Yu13}, let us consider $\H = \R$ and $f(x) = \frac{1}{2}x^2$ for all $x \in \R$.  Since $\prox_{\gamma f} = \frac{1}{1+\gamma} \I$ for all $\gamma \geq 0$ and $\prox^f_f = \frac{2}{3} \I$, we retrieve that
$$ \dfrac{1}{3} \I  =  \prox_{2f} =  \prox_{f+f} = \prox_f \circ \prox^f_f = \dfrac{1}{3} \I \neq \dfrac{1}{4} \I = \prox_f \circ \prox_f, $$
which illustrates Theorem~\ref{thm1}.
\end{example}

\subsection{Properties}

Let $f$, $g \in \Gamma_0 (\H)$. We know that $\prox^f_g$ is a generalization of $\prox_g$ in the sense that $\prox^f_g = \prox_g$ if $f$ is constant for instance. In the next proposition, our aim is to provide more general sufficient (and necessary) conditions under which $\prox^f_g = \prox_g$. We will base our discussion on the following conditions:
\begin{equation}\label{eqcondition2}\tag{$\mathrm{C}_2$}
\forall x \in \H, \quad \partial g(x) \subset \partial g (\prox_f (x) ),
\end{equation}
\begin{equation}\label{eqcondition3}\tag{$\mathrm{C}_3$}
\forall x \in \H , \quad \partial g( \prox_f (x) ) \subset \partial g ( x ).
\end{equation}
Note that Condition~\eqref{eqcondition2} has been introduced by Y.-L. Yu in \cite{Yu13} as a sufficient condition under which $\prox_{f+g} = \prox_f \circ \prox_g$.

\begin{proposition}\label{prop3}
Let $f$, $g \in \Gamma_0 (\H)$ with $\dom (f) \cap \dom (g) \neq \emptyset$.
\begin{enumerate}
\item[\rm{(i)}] If Condition~\eqref{eqcondition2} is satisfied, then $\prox_g (x) \in \prox^f_g (x)$ for all $x \in \H$.
\item[\rm{(ii)}] If Conditions~\eqref{eqcondition1} and \eqref{eqcondition3} are satisfied, then $\prox^f_g (x) = \prox_g (x)$ for all $x \in \H$.
\end{enumerate}
In both cases, Condition~\eqref{eqcondition1} is satisfied and the equality $\prox_{f+g} = \prox_f \circ \prox_g$ holds true.
\end{proposition}

\begin{proof}
Let $x \in \H$. If Condition~\eqref{eqcondition2} is satisfied, considering $y = \prox_g (x)$, we get that~$x \in y + \partial g(y) \subset y + \partial g (\prox_f (y))$ and thus $y \in \prox^f_g (x)$. In particular, it holds that $\D (\prox^f_g )=\H$ and thus Condition~\eqref{eqcondition1} is satisfied from Proposition~\ref{prop1}. Secondly, if Conditions~\eqref{eqcondition1} and~\eqref{eqcondition3} are satisfied, then~$\D ( \prox^f_g) = \H$ from Proposition~\ref{prop1}. Considering $y \in \prox^f_g (x)$, we get that $x \in  y + \partial g (\prox_f (y)) \subset y + \partial g(y)$ and thus $y = \prox_g (x)$. The last assertion of Proposition~\ref{prop3} directly follows from Theorem~\ref{thm1}.
\end{proof}

In the first item of Proposition~\ref{prop3} and if $\prox^f_g$ is set-valued, we are in the situation where $\prox_g$ is a selection of $\prox^f_g$. Proposition~\ref{prop4} specifies this selection in the case where $\partial (f+g) = \partial f + \partial g$.

\begin{lemma}\label{lem1}
Let $f$, $g \in \Gamma_0 (\H)$ with $\dom (f) \cap \dom (g) \neq \emptyset$. Then $\prox^f_g (x)$ is a nonempty closed and convex subset of~$\H$ for all $x \in \D (\prox^f_g)$.
\end{lemma}

\begin{proof}
The proof of Lemma~\ref{lem1} is provided after the proof of Proposition~\ref{prop5} (required). 
\end{proof}

\begin{proposition}\label{prop4}
Let $f$, $g \in \Gamma_0 (\H)$ with $\dom (f) \cap \dom (g) \neq \emptyset$ and such that $\partial (f+g) = \partial f + \partial g$ and let $x \in \H$. If~$\prox_g (x) \in \prox^f_g (x)$, then
$$ \prox_g (x) = \proj_{\prox^f_g (x)} ( \prox_{f+g} (x) ). $$
\end{proposition}

\begin{proof}
If $\prox_g (x) \in \prox^f_g (x)$, then $x \in \D (\prox^f_g)$ and thus $\prox^f_g (x)$ is a nonempty closed and convex subset of~$\H$ from Lemma~\ref{lem1}. Let $z \in \prox^f_g (x)$. In particular we have $\prox_f (z) = \prox_{f+g}(x)$ from Theorem~\ref{thm1}. Using the fact that $x-\prox_g(x) \in \partial g (\prox_g (x))$ and $x-z \in \partial g (\prox_f (z)) = \partial g ( \prox_{f+g} (x))$ together with the monotonicity of $\partial g$, we obtain that
\begin{equation*}
\langle \prox_{f+g} (x) - \prox_g (x) , z - \prox_g (x) \rangle 
= \langle \prox_{f+g} (x) - \prox_g (x) , (x- \prox_g (x) ) - (x-z) \rangle \leq 0 .
\end{equation*}
Since $\prox_g (x) \in \prox^f_g (x)$, we conclude the proof from the characterization of $\proj_{\prox^f_g (x)}$.
\end{proof}

\begin{remark}\label{remMNE}
Let $f = \iota_{\{ \omega \}}$ with $\omega \in \H$ and let $g \in \Gamma_0 (\H)$ such that $\omega \in \int (\dom ( g ))$. Hence the additivity condition $\partial (f+g) = \partial f + \partial g$ is satisfied from Remark~\ref{remdsum}. From Remark~\ref{remRR} and since $\prox_f =\proj_{ \{ \omega \}}$, we easily deduce that $\RR (\prox_{f+g}) = \{ \omega \}$. Let~$x \in \H$ such that $\prox_g (x) \in \prox^f_g (x)$. From Proposition~\ref{prop4} we get that
$$ \prox_g (x) = \proj_{\prox^f_g (x)} ( \omega ). $$
If moreover $\omega = 0$, we deduce that $\prox_g (x)$ is the particular selection that corresponds to the element of minimal norm in $\prox^f_g (x)$ (also known as the \textit{lazy selection}). The following example is in this sense.
\end{remark}

\begin{example}\label{ex4}
Let us consider the framework of Example~\ref{ex2}. In that case, Conditions~\eqref{eqcondition1} and~\eqref{eqcondition2} are satisfied. We deduce from Proposition~\ref{prop3} that $\prox_g(x) \in \prox^f_g(x)$ for all $x \in \R$. From Remark~\ref{remMNE}, we conclude that $\prox_g (x)$ is exactly the element of minimal norm in $\prox^f_g (x)$ for all $x \in \R$. This result is clearly illustrated by the graphs of $\prox_g$ and $\prox^f_g$ provided in Figure~\ref{fig1}.
\end{example}

Let $f$, $g \in \Gamma_0 (\H)$ with $\dom (f) \cap \dom (g) \neq \emptyset$ and such that $\partial (f+g) = \partial f + \partial g$. From Theorem~\ref{thm1}, one can easily see that, if~$\prox_f$ is injective, then $\prox^f_g$ is single-valued. Since the injection of $\prox_f$ is too restrictive, other sufficient conditions under which $\prox^f_g$ is single-valued are provided from Theorem~\ref{thm1} in the next proposition.

\begin{proposition}\label{prop2}
Let $f$, $g \in \Gamma_0 (\H)$ with $\dom (f) \cap \dom (g) \neq \emptyset$ and such that $\partial (f+g) = \partial f + \partial g$. If either~$\partial f$ or~$\partial g$ is single-valued, then $\prox^f_g$ is single-valued.
\end{proposition}

\begin{proof}
Let $x \in \H$ and let $z_1$, $z_2 \in \prox^f_g (x)$. From Theorem~\ref{thm1}, it holds that $\prox_f (z_1) = \prox_{f} (z_2) = \prox_{f+g} (x)$. If the operator $\partial f$ is single-valued, we obtain that $z_1 = \prox_{f+g} (x) + \partial f (\prox_{f+g} (x)) = z_2$. If the operator $\partial g$ is single-valued, we get $x-z_1 = \partial g ( \prox_f (z_1) ) = \partial g ( \prox_f (z_2) ) = x-z_2 $ and thus $z_1 = z_2$.
\end{proof}

\section{Relations with the Douglas-Rachford operator}\label{sec3}
Let $f$, $g \in \Gamma_0 (\H)$. The \textit{Douglas-Rachford operator} $\T_{f,g} : \H \to \H$ associated to $f$ and $g$ is usually defined by
$$ \T_{f,g} (y) :=  y-\prox_f (y) + \prox_g ( 2 \prox_f(y) - y ), $$
for all $y \in \H$. We refer for instance to~\cite[Section 28.3 p.517]{BauCom17} where details can be found on this classical operator.

\medskip

One aim of this section is to study the relations between the $f$-proximal operator~$\prox^f_g$ introduced in this paper and the Douglas-Rachford operator $\T_{f,g}$. For this purpose, we introduce an extension~$\TT_{f,g} : \H \times \H \to \H$ of the classical Douglas-Rachford operator defined by
$$ \TT_{f,g} (x,y) :=  y-\prox_f (y) + \prox_g ( x+\prox_f(y) - y ), $$
for all $x$, $y \in \H$. 

\medskip

Note that $ \T_{f,g} (y) = \TT_{f,g} ( \prox_f (y) , y ) $ for all $ y \in \H$, and that the definition of $\TT_{f,g}$ only depends on the knowledge of $\prox_f$ and $\prox_g$. 

\subsection{Several characterizations of $\prox^f_g$}

Let $f$, $g \in \Gamma_0 (\H)$. In this subsection, our aim is to derive several characterizations of $\prox^f_g$ in terms of solutions of variational inequalities, of minimization problems and of fixed point problems (see Proposition~\ref{prop5}).

\begin{lemma}\label{lem2}
Let $f$, $g \in \Gamma_0 (\H)$. It holds that
$$ \TT_{f,g} (x,\cdot) = \prox_{g^\ast \circ L_x} \circ \prox_{f^\ast} ,$$
for all $x \in \H$.
\end{lemma}

\begin{proof}
Let $x \in \H$. Lemma~\ref{lem2} directly follows from the equality $\prox_{g^\ast \circ L_x} = L_x \circ \prox_{g^\ast} \circ L_x$ (see~\cite[Proposition 24.8 p.416]{BauCom17}) and from Moreau's decompositions.
\end{proof}

\begin{proposition}\label{prop5}
Let $f$, $g \in \Gamma_0 (\H)$. It holds that
$$ \prox^f_g (x) = \Sol_{\VI} ( \prox_f , g^\ast \circ L_x ) = \argmin \, ( \M_{f^\ast} + g^\ast \circ L_x ) = \Fix (\TT_{f,g} (x,\cdot) ), $$
for all $x \in \H$.
\end{proposition}

\begin{proof}
In this proof we will use standard properties of convex analysis recalled in Section~\ref{secnot}. Let~$x \in \H$. One can easily prove that $\partial (g^\ast \circ L_x) = - \partial g^\ast \circ L_x$. For all~$y \in \H$, it holds that
\begin{eqnarray*}
y \in \prox^f_g (x) & \Longleftrightarrow & x-y \in \partial g ( \prox_f (y) ) \\
& \Longleftrightarrow & \prox_f (y) \in \partial g^\ast ( x-y ) \\
& \Longleftrightarrow &  -\prox_f (y) \in \partial (g^\ast \circ L_x) ( y ). 
\end{eqnarray*}
Moreover, since $\dom (\M_{f^\ast}) = \H$ and from Remark~\ref{remdsum}, we have
\begin{eqnarray*}
-\prox_f (y) \in \partial (g^\ast \circ L_x) ( y ) & \Longleftrightarrow & 0 \in \nabla \M_{f^\ast} (y) + \partial (g^\ast \circ L_x) ( y ) \\
& \Longleftrightarrow & 0 \in \partial ( \M_{f^\ast} + g^\ast \circ L_x )(y). 
\end{eqnarray*}
Finally,
\begin{eqnarray*}
-\prox_f (y) \in \partial (g^\ast \circ L_x) ( y ) & \Longleftrightarrow & \prox_{f^\ast} (y) \in y + \partial (g^\ast \circ L_x) (y) \\
& \Longleftrightarrow & y = \prox_{g^\ast \circ L_x} \circ \prox_{f^\ast} (y) .
\end{eqnarray*}
This concludes the proof from Lemma~\ref{lem2}. 
\end{proof}

\begin{proof}[Proof of Lemma~\ref{lem1}]
Let $x \in \D (\prox^f_g)$. In particular $\prox^f_g (x)$ is not empty. From Proposition~\ref{prop5}, we have
$$ \prox^f_g (x) = \argmin \, ( \M_{f^\ast} + g^\ast \circ L_x ). $$
Since $\M_{f^\ast} + g^\ast \circ L_x \in \Gamma_0 (\H)$, one can easily deduce that $\prox^f_g (x)$ is closed and convex.
\end{proof}

\subsection{A weakly convergent algorithm that computes $\prox^f_g$ numerically}\label{secnum}
Let $f$, $g \in \Gamma_0 (\H)$. In this section, our aim is to derive from Proposition~\ref{prop5} an algorithm, that depends only on the knowledge of $\prox_f$ and $\prox_g$, allowing to compute numerically an element of $\prox^f_g (x)$ for all~$x \in \D (\prox^f_g)$. We refer to Algorithm~\eqref{eqalgo} in Theorem~\ref{thm2}.

\medskip

Moreover, if the additivity condition $\partial (f+g) = \partial f + \partial g$ is satisfied, it follows from Theorem~\ref{thm1} that Algorithm~\eqref{eqalgo} is an algorithm allowing to compute numerically~$\prox_{f+g}(x) $ for all $x \in \H$ with the only knowledge of $\prox_f$ and $\prox_g$.

\begin{theorem}\label{thm2}
Let $f$, $g \in \Gamma_0 (\H)$ and let $x \in \D ( \prox^f_g)$ be fixed. Then, Algorithm~\eqref{eqalgo} given by
\begin{equation}\tag{$\mathcal{A}_1$}\label{eqalgo}
\left\lbrace
\begin{array}{l}
y_0 \in \H , \\[5pt]
y_{k+1} = \TT_{f,g} ( x , y_k ), 
\end{array}
 \right.
\end{equation}
weakly converges to an element $y^* \in \prox^f_g (x)$. Moreover, if $\dom (f) \cap \dom (g) \neq \emptyset$ and $\partial (f+g) = \partial f + \partial g$, it holds that $\prox_f (y^*) = \prox_{f+g} (x)$.
\end{theorem}

\begin{proof}
From Lemma~\ref{lem2}, $\TT_{f,g} (x,\cdot)$ coincides with the composition of two firmly non-expansive operators, and thus of two non-expansive and $\frac{1}{2}$-averaged operators (see~\cite[Remark~4.34(iii) p.81]{BauCom17}). Since $x \in \D ( \prox^f_g)$, it follows from Proposition~\ref{prop5} and Lemma~\ref{lem2} that $\Fix ( \prox_{g^\ast \circ L_x} \circ \prox_{f^\ast} ) \neq \emptyset$. We conclude from \cite[Theorem 5.23 p.100]{BauCom17} that Algorithm~\eqref{eqalgo} weakly converges to a fixed point $y^*$ of $\TT_{f,g} (x,\cdot)$. From Proposition~\ref{prop5}, it holds that $y^* \in \prox^f_g (x)$. Finally, if $\dom (f) \cap \dom (g) \neq \emptyset$ and $\partial (f+g) = \partial f + \partial g$, we conclude that $\prox_f (y^*) = \prox_{f+g} (x)$ from Theorem~\ref{thm1}.
\end{proof}

\begin{remark}\label{remcomb2bis}
As already mentioned in the introduction, it turns out that Algorithm~\eqref{eqalgo} was already considered, up to some translations, and implemented in previous works (see, e.g., the so-called \textit{dual forward-backward splitting} in \cite[Algorithm~3.5]{combdungvu}), showing that the $f$-proximal operator $\prox^f_g$ is already present (in a hidden form) and useful for numerical purposes in the existing literature. However, to the best of our knowledge, it has never been explicitly expressed in a closed formula such as~\eqref{eqdefdef} and neither been deeply studied from a theoretical point of view.
\end{remark}

\begin{remark}\label{remcomb2}
Let us discuss with more details the relationship between the present work and the one proposed in~\cite{combdungvu}. Let $x \in \H$. In \cite[Proposition~3.4]{combdungvu}, the authors prove that if $v \in \H$ is a solution to
\begin{equation}\label{eqopcomb}
\argmin \, ( \M_{f^\ast} \circ L_x + g^\ast  ) 
\end{equation}
then $\prox_f(x-v) = \prox_{f+g}(x)$. Combining this result with Proposition~\ref{prop5} easily constitutes an alternative proof of the new decomposition formula \eqref{eqdecomp} derived in this paper. Moreover, in \cite[Algorithm~3.5]{combdungvu}, the authors consider the so-called \textit{dual forward-backward splitting} given by
$$ v_{k+1} = \prox_{g^*} (v_k + \prox_f (x-v_k) ), $$
which is related to Algorithm~\eqref{eqalgo} by setting $y_k = x - v_k$. From Proposition~\ref{prop5}, the present work points out that the operator given in~\eqref{eqopcomb} actually coincides, up to a translation, with a generalization of the classical proximal operator, that is exactly the $f$-proximal operator introduced and studied from a theoretical point of view in this paper. In this section we also prove that Algorithm~\eqref{eqalgo} actually coincides with a fixed-point algorithm associated to a generalized version of the classical Douglas-Rachford operator.
\end{remark}

\begin{remark}
Let $f$, $g \in \Gamma_0 (\H)$ and let $x \in \D ( \prox^f_g)$. Algorithm~\eqref{eqalgo} consists in a fixed-point algorithm from the characterization given in Proposition~\ref{prop5} by
$$ \prox^f_g (x) = \Fix (\TT_{f,g} (x,\cdot) ). $$
Actually, one can easily see that Algorithm~\eqref{eqalgo} also coincides with the well-known \textit{Forward-Backward algorithm} {(see \cite[Section 10.3 p.191]{ComPes11} for details)} from the characterization given in Proposition~\ref{prop5} by
$$ \prox^f_g (x) = \argmin \, ( \M_{f^\ast} + g^\ast \circ L_x ) .$$
Indeed, we recall that $\M_{f^\ast}$ is differentiable with $\nabla \M_{f^\ast} = \prox_f$. We also refer to Section~\ref{sec41} for a brief discussion about the Forward-Backward algorithm.
\end{remark}

\subsection{Recovering a classical result from the decomposition formula}\label{sec33}
Let $f$, $g \in \Gamma_0 (\H)$ with $\dom (f) \cap \dom (g) \neq \emptyset$ and such that $\partial (f+g) = \partial f + \partial g$. Our aim in this section is to recover in a simple way, with the help of the decomposition formula~\eqref{eqdecomp}, the well-known equality
$$ \argmin \, f +g = \prox_f ( \Fix (\T_{f,g} ) ). $$
This result can be found for example in \cite[Proposition~26.1]{BauCom17}.

\begin{lemma}\label{lem3}
Let $f$, $g \in \Gamma_0(\H)$. It holds that
$$ \Fix (\T_{f,g} ) = \Fix ( \prox^f_g \circ \prox_f ). $$
\end{lemma}

\begin{proof}
Let $z \in \H$. It holds from Proposition~\ref{prop5} that
\begin{eqnarray*}
z  \in \Fix (\T_{f,g} ) & \Longleftrightarrow & z = \T_{f,g} (z) = \TT_{f,g} ( \prox_f (z) , z) \\
& \Longleftrightarrow & z \in \Fix ( \TT_{f,g} ( \prox_f (z) , \cdot ) ) = \prox^f_g ( \prox_f ( z ) ) \\
& \Longleftrightarrow & z \in \Fix ( \prox^f_g \circ \prox_f ). 
\end{eqnarray*}
The proof is complete.
\end{proof}

\begin{proposition}\label{proppourintro}
Let $f$, $g \in \Gamma_0 (\H)$ with $\dom (f) \cap \dom (g) \neq \emptyset$ and such that $\partial (f+g) = \partial f + \partial g$. It holds that
$$ \argmin \, f +g = \prox_f ( \Fix (\T_{f,g} ) ). $$
\end{proposition}

\begin{proof}
Let $y \in \Fix (\T_{f,g})$. Then $y \in \Fix ( \prox^f_g \circ \prox_f )$ from Lemma~\ref{lem3}. Thus $ y \in \prox^f_g \circ \prox_f (y)$. From the decomposition formula~\eqref{eqdecomp}, we get that $\prox_f (y) = \prox_{f+g} ( \prox_f (y) )$ and thus $\prox_f (y) \in \argmin \, f+g$.

\medskip

Let $x \in \argmin \, f +g$. Since $\D (\prox^f_g) = \H$ from Proposition~\ref{prop1}, let us consider $y \in \prox^f_g (x)$. From the decomposition formula~\eqref{eqdecomp}, it holds that $x=\prox_{f+g} (x) = \prox_f(y)$. Let us prove that $y \in \Fix (\T_{f,g})$. Since $y \in \prox^f_g (x) = \Fix ( \TT_{f,g} ( x , \cdot ) )$, we get that $y = \TT_{f,g} ( x , y) = \TT_{f,g} ( \prox_{f} (y) , y) = \T_{f,g} (y)$. The proof is complete.
\end{proof}

\section{Some other applications and forthcoming works}\label{sec4}

This section can be seen as a conclusion of the paper. Its aim is to provide a glimpse of some other applications of our main result (Theorem~\ref{thm1}) and to raise open questions for forthcoming works. This section is splitted into two parts.

\subsection{Relations with the Forward-Backward operator}\label{sec41}
Let $f$, $g \in \Gamma_0 (\H)$ such that $g$ is differentiable on $\H$. In that situation, note that the additivity condition~$\partial (f+g) = \partial f + \partial g$ is satisfied from Remark~\ref{remdsum}, and that Proposition~\ref{prop2} implies that $\prox^f_g$ is single-valued.

\medskip

In that framework, the classical \textit{Forward-Backward operator} $\BF_{f,g} : \H \to \H$ associated to $f$ and $g$ is usually defined by
$$ \BF_{f,g} (y) :=  \prox_f ( y - \nabla g (y) ), $$
for all $y \in \H$. {We refer to \cite[Section 10.3 p.191]{ComPes11} for more details.} Let us introduce the extension~$\BFF_{f,g} : \H \times \H \to \H$  defined by
$$ \BFF_{f,g} (x,y) := \prox_f ( x - \nabla g (y) ), $$
for all $x$, $y \in \H$. In particular, it holds that $ \BF_{f,g} (y) = \BFF_{f,g} ( y , y )$ for all $ y \in \H$. The following result follows from the decomposition formula~\eqref{eqdecomp} in Theorem~\ref{thm1}.

\begin{proposition}\label{prop6}
Let $f$, $g \in \Gamma_0 (\H)$ such that $g$ is differentiable on $\H$. Then
$$ \prox_{f+g} (x) = \Fix ( \BFF_{f,g} (x ,\cdot ) ), $$
for all $x \in \H$.
\end{proposition}

\begin{proof}
Let $x \in \H$. Firstly, let $z = \prox_{f+g} (x)$ and let $y = \prox^f_g (x)$. In particular, we have $x = y + \nabla g (\prox_f (y))$. From the decomposition formula~\eqref{eqdecomp}, we get that $z = \prox_f (y) = \prox_f ( x - \nabla g ( \prox_f (y) ) ) = \prox_f ( x - \nabla g ( z ) ) = \BFF_{f,g} (x,z) $. Conversely, let $z \in \Fix ( \BFF_{f,g} (x ,\cdot ) )$, that is, $z = \prox_f ( x - \nabla g ( z ) )$. Considering $y = x - \nabla g (z)$, we have $z = \prox_f (y)$ and thus $x=y+\nabla g( \prox_f(y))$, that is, $y = \prox^f_g (x)$. Finally, from the decomposition formula~\eqref{eqdecomp}, we get that $z = \prox_f \circ \prox^f_g (x) = \prox_{f+g} (x)$.
\end{proof}

From Proposition~\ref{prop6}, we retrieve the following classical result (see, e.g., \cite[Proposition~26.1]{BauCom17}).

\begin{proposition}
Let $f$, $g \in \Gamma_0 (\H)$ such that $g$ is differentiable on $\H$. Then
$$ \argmin \, f + g = \Fix ( \BF_{f,g} ). $$
\end{proposition}

\begin{proof}
Let $x \in \H$. It holds that
\begin{eqnarray*}
x \in \argmin \, f + g & \Longleftrightarrow & x = \prox_{f+g} (x) \\
& \Longleftrightarrow & x \in \Fix ( \BFF_{f,g} (x ,\cdot ) ) \\
& \Longleftrightarrow &  x = \BFF_{f,g} (x,x) = \BF_{f,g}(x) \\
& \Longleftrightarrow &  x \in \Fix (\BF_{f,g}).
\end{eqnarray*}
The proof is complete.
\end{proof}

Let $f$, $g \in \Gamma_0 (\H)$ such that $g$ is differentiable on $\H$. The classical \textit{Forward-Backward algorithm} $x_{n+1} = \BF_{f,g} (x_n)$ is a powerful tool since it provides an algorithm, only requiring the knowledge of $\prox_f$ and $\nabla g$, that weakly converges (under some conditions on $g$, see \cite[Section~28.5 p.522]{BauCom17} for details) to a fixed point of $\BF_{f,g}$, and thus to a minimizer of $f+g$.

\medskip

From Proposition~\ref{prop6}, and for all $x \in \H$, one can consider the algorithm (potentially weakly convergent) given by
\begin{equation}\tag{$\mathcal{A}_2$}\label{eqalgo3}
\left\lbrace
\begin{array}{l}
y_0 \in \H , \\[5pt]
y_{k+1} = \BFF_{f,g} ( x , y_k ), 
\end{array}
 \right.
\end{equation}
in order to compute numerically $\prox_{f+g}(x)$, with the only knowledge of $\prox_f$ and~$\nabla g$. Convergence proof (under some assumptions on $f$ and $g$) of Algorithm~\eqref{eqalgo3} should be the topic of a future work.

\subsection{Application to sensitivity analysis for variational inequalities}\label{sec42}
As a conclusion of the present paper, we turn back to our initial motivation, namely the sensitivity analysis, with respect to a nonnegative parameter $t \geq 0$, of some parameterized linear variational inequalities of second kind in a real Hilbert space $\H$. More precisely, for all $t \geq 0$, we consider the variational inequality which consists of finding $u(t) \in \K$ such that
\begin{equation*} 
\langle u(t) , z - u(t) \rangle + g(z) - g (u(t)) \geq \langle r(t) , z - u(t) \rangle ,
\end{equation*}
for all $z \in \K$, where $\K \subset \H$ is a nonempty closed and convex set of constraints, and where $g \in \Gamma_{0}(\H)$ and $r : \R^+ \to \H$ are assumed to be given and smooth enough. The above problem admits a unique solution given by
$$ u(t) = \prox_{f+g} (r(t)), $$
where $f = \iota_{\K}$ is the indicator function of $\K$.



\medskip

Our aim is to provide from Theorem~\ref{thm1} a simple and compact formula for the derivative $u'(0)$ under some assumptions (see Proposition~\ref{PropSensitivity} for details). Following the idea of F.~Mignot in~\cite{Mig76} (see also \cite[Theorem~2 p.620]{haraux}), we first introduce the following sets
\begin{eqnarray*}
O_v & := & \left\lbrace w \in \H \mid \exists \lambda >0 , \; \proj_{\K} (v) + \lambda w \in \K \right\rbrace \cap \left[ v - \proj_{\K} (v) \right]^{\perp} , \\
C_v & := & \cl \Big( \left\lbrace w \in \H \mid \exists \lambda >0 , \; \proj_{K} (v) + \lambda w \in \K \right\rbrace \Big) \cap \left[  v - \proj_{\K} (v) \right]^{\perp} ,
\end{eqnarray*}
for all $v \in \H$, where $\perp$ denotes the classical orthogonal of a set. 

\begin{proposition} \label{PropSensitivity}
Let $v(t) := r(t) - \nabla g (u(t))$ for all $t \in \R$. If the following conditions are satisfied:
\begin{enumerate}
\item[\rm{(i)}] $r$ is differentiable at $t=0$;
\item[\rm{(ii)}] $g$ is twice differentiable on $\H$;
\item[\rm{(iii)}] $O_{v(0)}$ is dense in $C_{v(0)}$;
\item[\rm{(iv)}] $u$ is differentiable at $t=0$;
\end{enumerate}
then the derivative~$u'(0)$ is given by
$$
u'(0) = \prox_{\varphi_f + \psi_g} ( r'(0) ) ,
$$
where $\varphi_f := \iota_{C_{v(0)}}$ and $\psi_g (x) := \frac{1}{2} \langle \D^2 g(u(0))(x),x \rangle$ for all $x \in \H$.
\end{proposition}

\begin{proof}
Note that $v$ is differentiable at $t = 0$ with
$$ v'(0) = r'(0) - \D^2 g (u(0))(u'(0)). $$
Note that $\prox^f_g$ is single-valued from Proposition~\ref{prop2} and Remark~\ref{remdsum}. From the decomposition formula~\eqref{eqdecomp} in Theorem~\ref{thm1}, one can easily obtain that
$$
v(t) = \prox^f_g (r(t)), \qquad \text{and thus} \qquad u(t) = \prox_{f} \circ \prox^f_{g} ( r(t) ) = \proj_{\K} (v(t)) ,
$$
for all $t \geq 0$. Since $O_{v(0)}$ is dense in $C_{v(0)}$, we use the asymptotic expansion of~F.~Mignot~\cite[Theorem~2.1 p.145]{Mig76} and we obtain that
$$
u'(0) = \proj_{C_{v(0)}} ( v'(0) ) .
$$
We deduce that
$$
v'(0) + \D^2 g(u(0)) \circ \proj_{C_{v(0)}} (v'(0)) = r'(0) .
$$
Since $g$ is convex and since $C_{v(0)}$ is a nonempty closed convex subset of $\H$, we deduce that $\varphi_f$, $\psi_g \in \Gamma_0(\H)$. Moreover $\partial (\varphi_f + \psi_g ) = \partial \varphi_f + \partial \psi_g$ from Remark~\ref{remdsum} and $\prox^{\varphi_f}_{\psi_g}$ is single-valued from Proposition~\ref{prop2}. It also should be noted that~$\nabla \psi_g = \D^2 g ( u(0) )$. As a consequence, we have obtained that
$$
v'(0) + \nabla \psi_g \circ \prox_{\varphi_f} (v'(0)) = r'(0),
$$
that is, $v'(0) = \prox^{\varphi_f}_{\psi_g} (r'(0))$. We conclude the proof from the equality $u'(0) = \prox_{\varphi_f} ( v'(0) )$ and from Theorem~\ref{thm1}.
\end{proof}

\begin{remark}
Proposition~\ref{PropSensitivity} provides an expression of $u'(0)$ in terms of the proximal operator of a sum of two proper, lower semicontinuous and convex functions. Hence, it could be numerically computed from Algorithm~\eqref{eqalgo}, requiring the knowledge of $\proj_{C_{v(0)}}$ and $\prox_{\psi_g}$. 
Alternatively, if the convergence is proved, one can also consider Algorithm~\eqref{eqalgo3} requiring the knowledge of $\proj_{C_{v(0)}}$ and $\nabla \psi_g = \D^2 g(u(0))$.
\end{remark}

\begin{remark}\label{remfinal}
The relaxations in special frameworks of the assumptions of Proposition~\ref{PropSensitivity} should be the subject of future works. 
In particular, it would be relevant to provide sufficient conditions ensuring that $u$ is differentiable at $t=0$. A promising idea in this sense is to invoke the concepts of \textit{twice epi-differentiability} and \textit{proto-differentiability} introduced by R.T.~Rockafellar in \cite{Rockepidiff,Rockproto}.
\end{remark}

The application of Proposition~\ref{PropSensitivity} in the context of some shape optimization problems with unilateral contact and friction is the subject of a forthcoming research paper (work in progress).

\appendix

\section{A nonexistence result for a closed formula}\label{app}
The aim of the present appendix is to prove that there is no closed formula, independent of $f$ and $g$, allowing to write $\prox_{f+g}$ as a linear combination of compositions of linear combinations of $\I$, $\prox_f$, $\prox_g$, $\prox_f^{-1}$ and $\prox_g^{-1}$. 

\medskip

For this purpose, let us introduce the elementary operator $\P^{\mu}_{f,g} : \H \rightrightarrows \H$ defined by
$$ \P^{\mu}_{f,g} := a \, \I + b \, \prox_f + c \, \prox_g + d \, \prox_f^{-1} + e \, \prox_g^{-1}, $$
for all $\mu = (a,b,c,d,e) \in \R^5$, all $f$, $g \in \Gamma_0(\H)$ and all Hilbert spaces $\H$. Let us assume by contradiction that for all Hilbert spaces $\H$, there exist $m$, $n \in \N^*$, $\lambda = (\lambda_i)_i \in \R^m$ and $(\mu_{ij})_{ij} \in (\R^5)^{m \times n}$ such that
\begin{equation}\label{eq654123}
 \prox_{f+g} = \di \sum_{i=1}^m \lambda_i \left( \prod_{j=1}^n \P^{\mu_{ij}}_{f,g} \right) ,
\end{equation}
for all $f$, $g \in \Gamma_0(\H)$, where $\prod$ denotes finite composition of operators.

\medskip

Then, let us consider the one-dimensional setting $\H = \R$ with $f(x) = g(x) = \frac{\gamma}{2} x^2$ for all $ x \in \R$ and all $\gamma \geq 0$. In that case $\prox_{f+g}$ is the linear function with slope $\frac{1}{1+2\gamma}$ and each $\P^{\mu_{ij}}_{f,g}$ is the linear function with slope
$$ \dfrac{(b_{ij}+c_{ij})+a_{ij}(1+\gamma) + (d_{ij}+e_{ij})(1+\gamma)^2}{1+\gamma}, $$
for all $\gamma \geq 0$. We deduce from Equality~\eqref{eq654123} that
$$ \dfrac{1}{1+2\gamma} = \di \sum_{i=1}^m \lambda_i \left( \prod_{j=1}^n \dfrac{(b_{ij}+c_{ij})+a_{ij}(1+\gamma) + (d_{ij}+e_{ij})(1+\gamma)^2}{1+\gamma} \right), $$
for all $\gamma \geq 0$, where $\prod$ denotes now the classical finite product of real numbers. We get that
$$ (1+\gamma)^n = (1+2\gamma) \di \sum_{i=1}^m \lambda_i \Big( \prod_{j=1}^n (b_{ij}+c_{ij})+a_{ij}(1+\gamma) + (d_{ij}+e_{ij})(1+\gamma)^2 \Big), $$
for all $\gamma \geq 0$. We easily deduce that the above polynomial equality can be extended to all $\gamma \in \R$, and thus it raises a contradiction for $\gamma = -\frac{1}{2}$.

\bibliographystyle{abbrv}

\begin{thebibliography}{10}

\bibitem{AttouchBrezis}
H.~Attouch and H.~Brezis.
\newblock Duality for the sum of convex functions in general {B}anach spaces.
\newblock In {\em Aspects of mathematics and its applications}, volume~34 of
  {\em North-Holland Math. Library}, pages 125--133. North-Holland, Amsterdam,
  1986.

\bibitem{BauCom17}
H.~H. Bauschke and P.~L. Combettes.
\newblock {\em Convex analysis and monotone operator theory in {H}ilbert
  spaces}.
\newblock CMS Books in Mathematics/Ouvrages de Math\'ematiques de la SMC.
  Springer, New York, 2017 (2nd edition).

\bibitem{BenLions}
A.~Bensoussan and J.-L. Lions.
\newblock In\'equations quasi-variationnelles d\'ependant d'un param\`etre.
\newblock {\em Ann. Scuola Norm. Sup. Pisa Cl. Sci. (4)}, 4(2):231--255, 1977.

\bibitem{combdungvu}
P.~L. Combettes, D.~D\~ung, and B.~C. V\~u.
\newblock Dualization of signal recovery problems.
\newblock {\em Set-Valued Var. Anal.}, 18(3-4):373--404, 2010.

\bibitem{ComPes11}
P.~L. Combettes and J.-C. Pesquet.
\newblock Proximal splitting methods in signal processing.
\newblock In {\em Fixed-point algorithms for inverse problems in science and
  engineering}, volume~49 of {\em Springer Optim. Appl.}, pages 185--212.
  Springer, New York, 2011.

\bibitem{DouglasRachford}
J.~Douglas and H.~H. Rachford.
\newblock On the numerical solution of heat conduction problems in two and
  three space variables.
\newblock {\em Trans. Amer. Math. Soc.}, 82:421--439, 1956.

\bibitem{EkelandTemam}
I.~Ekeland and R.~Temam.
\newblock {\em Analyse convexe et probl\`emes variationnels}.
\newblock Dunod; Gauthier-Villars, Paris-Brussels-Montreal, Que., 1974.
\newblock Collection \'Etudes Math\'ematiques.

\bibitem{haraux}
A.~Haraux.
\newblock How to differentiate the projection on a convex set in {H}ilbert
  space. {S}ome applications to variational inequalities.
\newblock {\em J. Math. Soc. Japan}, 29(4):615--631, 1977.

\bibitem{JBHU}
J.-B. Hiriart-Urruty and C.~Lemar\'echal.
\newblock {\em Fundamentals of convex analysis}.
\newblock Grundlehren Text Editions. Springer-Verlag, Berlin, 2001.

\bibitem{Mig76}
F.~Mignot.
\newblock Contr\^ole dans les in\'equations variationelles elliptiques.
\newblock {\em J. Functional Analysis}, 22(2):130--185, 1976.

\bibitem{moreau62bis}
J.-J. Moreau.
\newblock Fonctions convexes duales et points proximaux dans un espace
  hilbertien.
\newblock {\em C. R. Acad. Sci. Paris}, 255:2897--2899, 1962.

\bibitem{moreau65}
J.-J. Moreau.
\newblock Proximit\'e et dualit\'e dans un espace hilbertien.
\newblock {\em Bull. Soc. Math. France}, 93:273--299, 1965.

\bibitem{QiuMag}
Y.~Qiu and T.~L. Magnanti.
\newblock Sensitivity analysis for variational inequalities.
\newblock {\em Math. Oper. Res.}, 17(1):61--76, 1992.

\bibitem{Rock3}
R.~T. Rockafellar.
\newblock {\em Convex analysis}.
\newblock Princeton Mathematical Series, No. 28. Princeton University Press,
  Princeton, N.J., 1970.

\bibitem{Rockepidiff}
R.~T. Rockafellar.
\newblock Maximal monotone relations and the second derivatives of nonsmooth
  functions.
\newblock {\em Ann. Inst. H. Poincar\'e Anal. Non Lin\'eaire}, 2(3):167--184,
  1985.

\bibitem{Rockproto}
R.~T. Rockafellar.
\newblock Proto-differentiability of set-valued mappings and its applications
  in optimization.
\newblock {\em Ann. Inst. H. Poincar\'e Anal. Non Lin\'eaire}, 6:449--482,
  1989.
\newblock Analyse non lin\'eaire (Perpignan, 1987).

\bibitem{Shapiro}
A.~Shapiro.
\newblock Sensitivity analysis of parameterized variational inequalities.
\newblock {\em Math. Oper. Res.}, 30(1):109--126, 2005.

\bibitem{Yu13}
Y.-L. Yu.
\newblock On decomposing the proximal map.
\newblock {\em Advances in Neural Information Processing Systems}, 26:91--99,
  2013.

\end{thebibliography}

\end{document}